\documentclass[psamsfonts]{amsart}

\usepackage{amssymb,amsfonts}
\usepackage[all,arc]{xy}
\usepackage{enumerate}
\usepackage{mathrsfs}

\newcommand{\Rr}{\mathbb{R}}

\newtheorem{thm}{Theorem}

\newtheorem{prop}{Proposition}
\newtheorem{defi}{Definition}
\newtheorem{rema}{Remark}
\newtheorem{coro}{Corollary}
\newtheorem{lem}{Lemma}
\newtheorem{exm}{Example}

\makeatletter
\let\c@equation\c@thm
\makeatother
\numberwithin{equation}{section}

\title{Dimension preserving resolutions of singularities of Poisson structures}

\author{Hichem Lassoued}

\address{Laboratoire d'Analyse Numerique et d'Optimisation et de Statistique L.A.N.O.S,  Universit\'e Badji Mokhtar -Annaba- B.P.12 Annaba, 23000 Alg\'erie.
 }

\address{Institut Elie Cartan de Lorraine I.E.C.L., UMR 7502 , Universit\'e de Lorraine,
\^Ile du Saulcy 57010 Metz, France  }

\date{March 10th, 2017}

\begin{document}

\begin{abstract} Some Poisson structures do admit resolutions by symplectic manifolds of the same dimension.
 We give examples and simple conditions under which such resolutions can not exist.
\end{abstract}

\maketitle

\tableofcontents

\section*{Introduction}

Poisson manifolds of dimension $n$ are known to admit symplectic realizations of dimension $2n$ (see \cite{DZ} or \cite{BX16} for the holomorphic case).
Recall that a \emph{symplectic realization} of a Poisson manifold $(M,\pi_M)$ is a triple of $(\Sigma,\Pi_\Sigma,\varphi)$ where $(\Sigma,\Pi_\Sigma)$ is a symplectic\footnote{For us, a symplectic manifold is a
Poisson manifold with an invertible bivector field.} manifold 
and $\varphi : \Sigma \to M$ is a surjective submersion on $(M,\pi_M)$ which is also a Poisson morphism. 

It is obviously impossible to find a symplectic realization of a Poisson manifold of dimension $n=dim M$, unless $M$ is itself symplectic. 
But it is possible to modify the concept of realization by imposing only that $\varphi$ is a surjective map, but not necessarily a submersion.
More precisely, we define symplectic resolutions as follows:

\begin{defi}\label{def:resolutions} 
Let $(M,\pi_M)$ be a real or complex  Poisson manifold of dimension $n$. We call \emph{symplectic resolution} a triple
$(\Sigma,\pi_\Sigma,\varphi)$ where $(\Sigma,\Pi_\Sigma)$ is a symplectic manifold of dimension $n$ and $\varphi : \Sigma \to M$ is a surjective Poisson morphism.
\end{defi}

The term "symplectic resolution" has been used by Arnaud Beauville in the context of the algebraic geometry \cite{sym} and BaoHua Fu \cite{res} for instance
- see for instance \cite{BellamySchedler} for lists of examples of those.
The concept that we have introduced in above is consistent with this previously given meaning.
It can not be really compared,  because we work with in differential geometry, while they work in algebraic geometry,
and singular points are for these authors points where the variety is singular, while for us singular points are those where the Poisson structure is singular. 
But the Poisson structures they resolve are symplectic at regular points, and their resolutions are birational at these points.
Something similar happens here:
we are going to see that the necessary condition for the existence of a symplectic resolution is that the bivector $\pi$ should be invertible
(i.e. symplectic) on an open dense subset of $M$ (see Proposition \ref{prop1}) and our resolutions are local diffeomorphisms at these points.
As a consequence, it makes sense to use the same name.

Many examples of symplectic resolutions exist 
(see Section \ref{sec1}). But we give in this article a simple example of Poisson structure of dimension $2$ which does not have reasonable symplectic resolutions.
 We use this result to exclude a large class of symplectic Poisson manifold which, although there are symplectic on a dense open subset, 
 do not admit reasonable symplectic resolutions. In particular, in the holomorphic setting, we prove that such connected resolutions do not exist
 for non-symplectic Poisson manifolds. 
 
\bigskip
\noindent {\textbf{ Conventions:}} Throughout this article, we denote the Poisson structures by $\pi$ or $\left\{\cdot,\cdot\right\}$ indifferently. We write 
$\pi$ when we see it as a section of $\wedge^2 TM$ and by $\left\{\cdot,\cdot\right\}$ the associated skew-symmetric biderivation of the algebra of smooth or holomorphic
functions on $M$. 
Also, we shall denote by $\pi_X$ or $ \left\{\cdot,\cdot\right\}_X$ a Poisson structure defined on the manifold $X$. 
As we shall in general not  consider two Poisson structures on the same manifold, this notation is not ambiguous. 
When the Poisson structure is invertible, that is, a symplectic Poisson structure, we shall use a capital letter instead, e.g. $\Pi_X$.

\section{Examples of symplectic resolutions of the same dimension in dimension~2}
\label{sec1}

In this section, we give examples of manifolds of dimension $2$ that do admit symplectic resolutions.
Consider the Poisson structure on $M:={\mathbb R}^2$ (or an open subset of ${\mathbb R}^2$) given by: 
\begin{equation}   
\label{laplussimple01}
\left\{ x,y \right\}_{M}=f(x,y)
\end{equation}
\noindent where $x,y$ are the canonical coordinates to $M$, and $f(x,y)$ is a smooth function. 
We suppose that the Poisson manifold $(M,\pi_M)$ described by (\ref{laplussimple01}) admits a symplectic resolution $(\Sigma,\Pi_\Sigma,\varphi)$. 
With $\Pi_{\Sigma}$ the symplectic Poisson structure on $\Sigma$. Since $ M= {\mathbb R}^2$, the map $\varphi$ is of the form $\varphi=(u,v)$,
where $u$ and $v$ are two smooth functions on $\Sigma $ with values in ${\mathbb R}$.
For any choice $(p,q)$ of local coordinates on $(\Sigma,\Pi_{\Sigma})$, the functions $u$ and $v$ are functions of the variables $p$ and $q$.

\begin{prop} 
\label{prop1}
Let $\Sigma $ be a manifold, $\Pi_\Sigma$ a symplectic Poisson structure on $\Sigma$, and $\varphi:\Sigma \to M = {\mathbb R}^2$ a surjective map, 
the pair $(\Sigma,\Pi_\Sigma,\varphi)$ is a symplectic resolution of the Poisson manifold $(M,\pi_M)$ which is given by (\ref{laplussimple01}) if and only if 
\begin{equation}\label{eq:conditionMorphisme}  \{p,q\}_\Sigma  \left(\frac{\partial u}{\partial p}\frac{\partial v}{\partial q}-\frac{\partial u}{\partial q}\frac{\partial v}{\partial p}\right)= f(u,v)
 \end{equation}
\noindent where $p, q$ are local coordinates on $(\Sigma,\Pi_{\Sigma})$ and
$\varphi=(u,v)$.
\end{prop}
\begin{proof}[Proof]
 The map $\varphi$ is Poisson if and only if $ \{\varphi^* x , \varphi^* y\}_\Sigma = \varphi^* \{x,y\} = \varphi^* f(x,y)$, with $(x,y)$
 the coordinates on $M ={\mathbb R}^2$. 
 As $u = \varphi^* x $ et $ v=\varphi^* y$, this condition is equivalent to $ \{u , v\}_\Sigma =  f(u,v)$, the result is obtained by writing in the coordinates $(p,q)$ 
 the bracket $\{ u,v\}_\Sigma $.
\end{proof}

\begin{exm}\normalfont
\label{ex:squares}
This example has appeared  in \cite{c1234}, Section 6.
We construct a symplectic resolution of the Poisson structure on $M:=\Rr^{2}$ given by~: 
\begin{equation}\label{h1} \left\{ x,y \right\}_{M} =x^{2}+y^{2},
\end{equation}
\noindent where $x,y$ denote the canonical coordinates of $M= \Rr^{2}$. Our candidate of symplectic resolution is given by $\Sigma := {\Rr}^{2}$ equipped 
with the canonical symplectic Poisson bracket $ \{q,p\}_ \Sigma =1 $, 
with $(p,q)$ the coordinates on $\Sigma := {\Rr}^2$ 
(we use different letters for coordinates on $M$ and $\Sigma$ which by accident, in our example, happen to be isomorphic). 
We define a map $\varphi$ from $\Sigma $ to $M $ by:
$$\begin{array}{rrcl} \varphi :&  \Sigma& \to & M \\ & (p,q)& \mapsto& (q\sin(pq),q\cos(pq)). 
   \end{array}
 $$

It is easy to check that $\varphi$ is surjective. We are left with the task of showing that this map $\varphi$ is a Poisson morphism, for which it suffices to check that the condition
given by Equation (\ref{eq:conditionMorphisme}) is satisfied.
 This is done by direct computation: as $u(p,q)=q\sin(pq)$ and $v(p,q)=q\cos(pq)$, we have on the one hand $f(u,v) = (q\sin(pq))^2+(q\cos(pq))^2 = q^2 $, 
 and on the other hand: 
\begin{eqnarray*}  \frac{\partial u}{\partial p}\frac{\partial v}{\partial q}-\frac{\partial u}{\partial q}\frac{\partial v}{\partial p}   &=&  
  \frac{\partial q\sin(pq) }{\partial p}\frac{\partial q\cos(pq)}{\partial q}-\frac{\partial q\sin(pq)}{\partial q}\frac{\partial q\cos(pq)}{\partial p}\\
  &=&     q^2 (\sin^2(pq)+\cos^2(pq)) = q^2.\\
 \end{eqnarray*}
 Since $\left\{p,q\right\}_{\Sigma}=1$ and $f(u,v)=q^2$. This proves the claim.
\end{exm}

We can also construct a symplectic resolution of a Poisson structure  more general  than (\ref{h1}).

\begin{exm}\normalfont
\label{ex:moregeneraleverpowers}
We equip $\Rr^{2}:= M$ with the Poisson structure defined by 
$$\left\{x,y\right\}_{M}=x^{2n}+y^{2m}, \hbox{ with $n \geq m > 1$}$$
where the couple $(x,y)$ stands for the canonical coordinates of $M$, and $n$ and $m$ are integers with $n \geq m$. Our candidate of symplectic resolution
is given by $\Sigma := \Rr^{2}$ with the Poisson bracket $\left\{p,q\right\}_{\Sigma}=q^{2n-2m}\sin^2(pq^{2m-1})+\cos^2(pq^{2m-1})$,
with $(p,q)$ being the canonical coordinates of $\Sigma$.
This Poisson bracket is symplectic, because  $q^{2n-2m}\sin^2(pq^{2m-1})+\cos^2(pq^{2m-1})$ is strictly positive an for all $p,q \in {\mathbb R}$.
We define a map $\varphi$ from $\Sigma$ to $M$ by:
$$\begin{array}{rrcl} \varphi :&  \Sigma& \to & M \\ & (p,q)& \mapsto& (q\sin(pq^{2m-1}),q\cos(pq^{2m-1})). 
   \end{array}
 $$
To show that $(\Sigma,\varphi)$ is a symplectic resolution of $M$, it is sufficient to check that the condition given by Equation (\ref{eq:conditionMorphisme}) is satisfied. 
As $u(p,q)=q\sin(pq^{2m-1})$ and $v(p,q)=q\cos(pq^{2m-1})$, we have on the one hand, $f(u,v) = (q\sin(pq))^{2n}+(q\cos(pq))^{2m} $, and on the other hand, by direct computation: 
\begin{eqnarray*}  \frac{\partial u}{\partial p}\frac{\partial v}{\partial q}-\frac{\partial u}{\partial q}\frac{\partial v}{\partial p}   &=&  
  \frac{\partial q\sin(pq^{2m-1}) }{\partial p}\frac{\partial q\cos(pq^{2m-1})}{\partial q} \\ & & -\frac{\partial q\sin(pq^{2m-1})}{\partial q}\frac{\partial q\cos(pq^{2m-1})}{\partial p}\\
  &=&     q^{2m} (\sin^2(pq^{2m-1})+\cos^2(pq^{2m-1})) = q^{2m}.\\
 \end{eqnarray*}
 This proves the claim. Since ${p,q}_{\Sigma}=1$ and $F(u,v)=q^{2m}$, this implies that any Poisson structure of the form 
 $$\left\{x,y\right\}_{M}=\kappa(x,y) (x^{2n}+y^{2m}),$$
 admits a resolution, provided that the function $\kappa(x,y)$ is strictly positive on $M$: it suffices to consider the triple $(\Sigma,(\varphi^* \kappa) \, \Pi_\Sigma,\varphi)$.
\end{exm}

\begin{exm}\normalfont
Examples \ref{ex:squares} and \ref{ex:moregeneraleverpowers} are Poisson structures with isolated singularities. We want to introduce an example with non-isolated singularities.
We equip $M:=\Rr^{2}$ with the bracket given in the canonical coordinates by $\left\{x,y\right\}_M=x$.
 We define $\Sigma  := \Rr^{2}\coprod{\Rr^{2}}\coprod{\Rr^{2}}$ to be the disjoint union of three copies of $\Rr^2$. 
There is a natural symplectic structure on $\Sigma $ obtained by restricting each of the three copies canonical symplectic Poisson structure. 
Consider the map $\varphi : \Sigma \rightarrow M$ defined by : 
\begin{equation*}
  \varphi :=\left\{\begin{array}{rcl}
              (\exp(p),q)   \hbox{ on the first copy } \\
              (-\exp(p),q)  \hbox{ on the second copy } \\
          (0,q)         \hbox{ on the third copy}   
 \end{array}\right.
\end{equation*}
A direct computation using Condition (\ref{eq:conditionMorphisme}) implies that the restriction on all three copies is a Poisson morphism, 
which is sufficient to make $\varphi$ a Poisson morphism. The map $\varphi$ is surjective, making the couple $(\Sigma,\varphi)$ a symplectic resolution  of $(M,\pi)$.
Note that $\Sigma$ is \emph{not} connected.
\end{exm}

\begin{exm}\normalfont
\label{ex:tropfacile}
There is a general but non-satisfying construction for making a symplectic resolution of any Poisson manifold of even dimension.
Let $(M,\pi_M)$ be a Poisson manifold of dimension $2d$. 
For all symplectic leaf ${\mathcal S}$ of $\pi_M$, consider the direct
product  $\Sigma_{\mathcal S} := {\mathcal S} \times {\mathbb R}^{2d - 2s}$ where $2s$ is the dimension of ${\mathcal S}$.
Equip  $\Sigma_{\mathcal S}$ with  the direct product Poisson symplectic structure $\Pi_{\Sigma_{\mathcal S} }$, with the understanding
that ${\mathcal S} $, the symplectic leaf,
comes equipped with its induced symplectic structure, and the vector space of even dimension $ {\mathbb R}^{2d - 2s}$ comes equipped with any symplectic Poisson structure. 
The natural map $\varphi_{\mathcal S}$ obtained by projecting ${\mathcal S} \times {\mathbb R}^{2d - 2s}$ onto ${\mathcal S}$, then
by including $ {\mathcal S}$ into $M$ is a Poisson map.

Now, let $S$ be the set of all symplectic leaves.
Let $\Sigma := \coprod_{{\mathcal S} \in S} \Sigma_{\mathcal S}$.
All connected components of this manifold have dimension $2d$ and
are symplectic manifolds, hence $\Sigma$ is symplectic, with respect to a symplectic Poisson structure $\Pi_\Sigma$
whose restriction to ${\Sigma_{\mathcal S} }$ is $\Pi_{\Sigma_{\mathcal S} }$. The map $\varphi: \Sigma \to M$
whose restriction to ${\mathcal S} \in S$ is $ \varphi_{\mathcal S}$ is a surjective Poisson map.
Hence, $(\Sigma,\Pi_\Sigma,\varphi)$ defines a Poisson resolution.
In general, the symplectic leaves form a non-countable family hence $\Sigma$ may not be separable (i.e. does not admit a dense countable subset).
\end{exm}

\section{Of Poisson structures that do not admit symplectic resolutions of the same dimension}

We describe in this section broad classes of Poisson manifolds that can not admit reasonable symplectic resolutions.
As shown in Example \ref{ex:tropfacile}, it is very reasonable to assume that the symplectic resolutions are separable, since
without such a condition they always exists, but they are not very interesting. The first result is the following.

\begin{prop} \label{nari}
A Poisson manifold $(M,\pi)$ that admits a separable symplectic resolution is symplectic on an open dense subset.
\end{prop}
\begin{proof}[Proof]
For all $\sigma \in \Sigma$ such that $d_\sigma \varphi $ is surjective, hence invertible,  $\varphi$ is a local diffeomorphism
from a neighborhood of $\sigma$ in $\Sigma$ to a neighborhood of $\varphi (\sigma) $ in $M$.
Hence, the bivector field $\pi_M$ is symplectic at the point $\varphi (\sigma) $.

Now, since $\varphi$ is surjective, it follows from Sard's Theorem (which is precisely
valid for smooth maps between separable manifolds) that the image through $\varphi$ of the set of points in $\Sigma$
where the differential of $\varphi$ is surjective (and therefore is a local diffeomorphism) is an open dense subset of $M$. 
This completes the proof.
\end{proof}

For Poisson manifolds that are symplectic on an open dense subset, the singular points of the Poisson structures
are exactly the points where the bivector field is not invertible. The following proposition will be used several times.
A critical value of $\varphi$ is a point $m \in M$ where $\sigma\in\Sigma$ with $\varphi(\sigma)=m$ and $d_{\sigma}\varphi$ is not surjective. A singular point of a Poisson structure is point $m$ where $\pi^{\#}_m:T_{m}^* M\rightarrow T_{m}M$ is not of maximal rank.
  
\begin{prop}
Let $(\Sigma,\pi_\Sigma,\varphi)$ be a symplectic resolution of the Poisson manifold $(M,\pi_M)$.
\begin{enumerate}
 \item Singular points of $\pi_M$ coincides with the set of critical values of $\varphi$.
 \item The differential $T_{\sigma}\varphi$ of $\varphi$ at a point $\sigma \in\Sigma$ is invertible 
if and only if $\varphi(\sigma)$ is a regular point of $\pi_M$.
\end{enumerate}
\end{prop}
\begin{proof}[Proofs]
The map $\varphi$ is a Poisson map if and only if the following diagram is commutative
for all $m \in M$ and all $\sigma \in \Sigma$ with $\varphi (\sigma) = m$:                     
$$ \xymatrix{ T_{\sigma}^* \Sigma  \ar[r]^{\Pi^{\#}_{\sigma}}& T_{\sigma} \Sigma \ar[d]^{T_{\sigma} \varphi} \\  T_{m}^* M \ar[u]^{T_{\sigma}^* \varphi} \ar[r]^{\pi^{\#}_m}  & T_{m} M.}$$
The commutativity of this diagram implies that  $\pi^{\#}_{m}$ is invertible if and only if vertical arrows are invertible, 
i.e. if and only if  $T_{\sigma}\varphi$ is invertible for all $\sigma$ in the inverse image of $m$ through $\varphi$.
This proves the result.
\end{proof}

Consider  the Poisson structure given by:
    
\begin{equation}\label{laplussimple}\left\{x,y\right\}_{M}=f(x,y)=xg(x,y),
\end{equation}
where $x,y$ are local coordinates on $M:=\Rr^{2}$ and $g(x,y)$ is a smooth function that does not vanish unless $x=0$. We saw in Section \ref{sec1} that, for $g=1$,
this Poisson structure admits a symplectic resolution for which $\Sigma$ is a disjoint union of three copies of ${\mathbb R}^2$. 
But we now show that $(M,\pi_M)$ admits no symplectic resolution when we impose  additional conditions. A resolution is said type proper when $\varphi$ is a \textit{proper} map,  i.e. the inverse image though $\varphi$ of a compact subset of $M$ is a compact subset.

\begin{thm}\label{lepremierthéorème}
The Poisson manifold $(M,\pi)$ described in (\ref{laplussimple}) admits no proper symplectic resolution. 
It does not admit connected real analytic or holomorphic symplectic resolution.
\end{thm}

Of course, when we say "It does not admit real analytic or holomorphic symplectic resolution", we mean that we consider the case where $f(x,y)$ is real analytic or holomorphic, 
and in the holomorphic case, we place ourselves on an open subset of $M\subset{\mathbb C}^2$ a open subset that contains $(0,0)$. We define resolutions in this context as being resolutions
in the previous sense, with $(\Sigma,\Pi_\Sigma,\varphi)$ real analytic or holomorphic.
 
\begin{rema}
The following problem is still open: in the smooth case, does the Poisson manifold $(M,\pi)$ described by (\ref{laplussimple}) admit a connected symplectic resolution?
\end{rema}

Let us prove Theorem \ref{lepremierthéorème}.

\begin{proof}[Proof]
Let $(\Sigma,\Pi_\Sigma,\varphi)$ be a symplectic resolution of $(M,\pi_{M})$. Since $M \simeq {\mathbb R}^2 $, we write $\varphi = (u,v)$
with $u,v$ smooth real valued functions on $\Sigma$. 
According to Sard's theorem \cite{Sard1}, applied to the differentiable function
$v : \Sigma \rightarrow \Rr$, function which is surjective since $\varphi$ is surjective, 
the set of critical values of $v$ admits a dense complement in $\Rr$. 

Let $v_0\in\Rr$ be outside the set of all critical values of $v$. Since $v_{0}$ is not a critical value, 
the inverse image by $v : \Sigma\rightarrow \Rr$ of $q_0$ is an union $({\mathcal C}_{i})_{i\in I}$ of curves.
Since $v$ is a submersion in a neighborhood of all point in $v^{-1}(v_0)$, the curves $({\mathcal C}_i)_{i \in I}$ whose union form $v^{-1}(v_0))$ can be separated:
that is to say there are open sets $(U_i)_{i \in I}$, with $U_i$ containing ${\mathcal C}_i$ for all indice $i\in I$, such that $U_i \cap U_j = \emptyset$
for all distinct $i,j \in I$. 

Let us now call curves of \textit{good type} curves in the previous set on there exists at least one point $\sigma$
with $u(\sigma) \in [-1,1]$. We claim that there are only finitely many curves of the good type.
Let $K$ be the inverse image through $\varphi$ of $ [-1,1]  \times \{v_0\} \subset M$.
Curves of the good type are those that intersect $K$.
Since $\varphi$ is proper, this inverse image $K$ is compact.
But the open subsets $ (U_i \cap K)_{i \in I}$ form a partition of $K$ by open subsets.
Now, in a partition of a compact set by open subset, only finitely many can be non-empty.
Said otherwise, there is a finite subset $j \in J$ such that $\cup_{j\in J}C_{J}$ contains the inverse image of $ [-1,1]  \times \{v_0\} \subset M$ through $ \varphi$.
Said otherwise, there are finitely many curves of the good type.

Consider the point $(0,v_0)\in M := \Rr^{2}$. 
This point is a singular point of $\pi_M$ by the definition thereof. 
Consider a point $\sigma\in \Sigma$ in its inverse image. This point $\sigma$ belongs to a curve ${\mathcal C}_{j_0} $ for a 
certain $j_0 \in I$. We are going to show that the image of ${\mathcal C}_{j_0}$ by $\varphi$ is reduced to the point $(0,v_0)$, or, equivalently, that
 $u$ identically vanishes on this curve ${\mathcal C}_{j_0}$.

Since the function $ v$ has a differential that does not vanish at the point $\sigma$, there exists 
another local function $p$, defined on a neighborhood of $\sigma$, such that the pair $(p,v)$ form local Darboux coordinates on an open subset 
$U_{\sigma}\subset\Sigma$. On $U_{\sigma}\subset\Sigma$ the map
 $\varphi$ reads  $\varphi : (p,q)\rightarrow (u(p,v),v)$ while the restriction $U_{\sigma}\subset\Sigma$ of the curve ${\mathcal C}_{i_0}$ is given by $v=v_0$.
By Proposition \ref{prop1}, since the map $\varphi$ is a Poisson morphism, the following differential equation in the variable $p$ holds for all value of $v$
in a neighborhood of $v_0$:
\begin{equation}\label{hih1} \frac{\partial u}{\partial p}(p,v)=f(u(p,v),v).
\end{equation} 
In particular, for $v=v_0$, we obtain the differential equation:
\begin{equation}
  \left\{\begin{array}{rcl}
              \frac{\partial u}{\partial p}(p,v_0)=f(u(p,v_0),v_0)\\
               \hbox{ and }  u(p_{0},v_{0})=0.
 \end{array}\right.
\end{equation} 
where $(p_{0},v_{0})$ are coordinates of the point $\sigma$. Since $f(0,v)=0$ for all value of $v$ and in particular $f(0,v_0)=0$, the Cauchy-Lipschitz
theorem imply that the differential equation (\ref{hih1}) admits for solution the zero function, i.e. $u(p,v_{0})=0$ for all $p$ close to zero.
In particular, the function $u$ vanishes on the restriction of the curve ${\mathcal C}_{j_0}$ to a neighborhood of $\sigma \in \Sigma$.
One can repeat the above deduction in a neighborhood of every point in ${\mathcal C}_{j_0}$, 
which proves that the function $u$ vanishes identically on the whole curve ${\mathcal C}_{j_0}$.

As previously stated, the inverse image by $v : \Sigma\rightarrow \Rr$ of $v_0$ is a union of curves.
The conclusion of the previous lines is that there are two types of such curves, those on which the restriction of the function $u$ is never equal to $0$ 
(curves that we call curves of the \textit{first type}) and those where $u$ vanishes identically, (curves that we call curves of the \textit{second type}). 
Since the map $\varphi$ is a surjection, there is necessarily at least one curve of the first type  is of the good type. Also all curves of the second type are of the good type.

Since there are only finitely many curves of the good type, those which are both of the good type and of the first type and those the second type
can be separated by two open subsets $V$ and $W$ of $\Sigma$.
 We now consider a sequence $(y_{n})_{n\in \mathbb{N}}\in \Sigma$ such that $\varphi(y_{n})=(\frac{1}{n},v_0)$. For all $n\geq 1$ Since the map $\varphi$ is proper, out of the sequence $y_{n}$, we can extract a convergent subsequence. 
 Let $\tilde{y}\in\Sigma$ be its limit. By construction, $\varphi(\tilde{y})=(0,v_0)$. For any $n\in \mathbb{N}$, the element $y_{n}$ belongs to a curve of the first type
 and of the good type,
 and therefore is in $V$, but its limit has to belong to a curve of the good type and of the second type, and is therefore in $W$. 
 This contradicts the assumption that $V\cap W = \emptyset$, this completes the proof.
\end{proof}

We now look at the real analytic or holomorphic case (i.e, the case where the function $f$ in (\ref{laplussimple}) is a real analytic or holomorphic function, 
and we impose that so are
$\Sigma,\Pi_\Sigma $ and $\varphi$. 
We will start with a lemma:
\begin{lem} \label{hch2}

Let $(\Sigma,\Pi,\varphi)$ be a connected symplectic resolution of the Poisson structure in (\ref{laplussimple}). For any point $\sigma\in\Sigma$ such as $d_{\sigma}v\neq 0$ et $\varphi(\sigma)\in\mathbf{M}_{sing}$,
there exists a neighborhood $U_{\sigma}$ of $\sigma$ such as $\varphi(U_{\sigma})\cap \mathbf{M}_{sing}=\left\{\varphi(\sigma)\right\}$ where $\mathbf{M}_{sing}\subset M$ is the set $\left\{x=0\right\}$, i.e. is the singular set of $\pi_{M}$.
\end{lem}
\begin{proof}[Proofs]

For any point $\sigma\in\Sigma$ such that $d_{\sigma} v\neq 0$, there exists a neighborhood $U_{\sigma}$ of $\sigma$ in $\Sigma$ and a function $p$ defined on $\Sigma$ such that the pair $(p,v)$ are local Darboux coordinates on $\Sigma$. the map $\varphi$ reads in these coordinates as : 
$$\begin{array}{rrcl} \varphi :&  U_{\sigma}\subset\Sigma& \to & M \\ & (p,q)& \mapsto& (u(p,q),q) 
   \end{array}
 $$
   
Let $(p_0,v_0)$ be the coordinates of a point $\sigma \in \Sigma$ such that $\varphi(\sigma)=(0,v_0) \in \mathbf{M}_{sing}$.
Proposition \ref{prop1} implies that Relations (\ref{hih1}) are satisfied.
Consider the function $h:v\rightarrow u(p_{0},v)$. Since the function $h$ is real analytic or holomorphic, and $h$ has a zero at $v=v_{0}$, 
there are two possibilities : either this function is  identically equal to zero or it has an isolated zero at $v_{0}$. Equation (\ref{hih1}) 
and Cauchy-Lipschitz theorem give that if $h$ is identically equal to $0$ on a neighborhood of $v_{0}$, 
the function $u$ vanishes at all point of $U_{\sigma}$. Since the function $u$ is real analytic or holomorphic,
this implies that $u=0$ on the whole connected manifold $\Sigma$.
But this is impossible because this contradicts the surjectivity of $\varphi$. Hence $v_{0}$ is necessarily an isolated zero of $h$ 
and Equation (\ref{hih1}) and Cauchy-Lipschitz theorem imply that $u(p,v)$ can not be zero for $v \neq v_{0}$ while $u(p,v_{0})=0$ for all $p$ close to $p_0$. 
More precisely, there exists $\eta',\eta''\in\Rr_+$ such that the inequalities $\left|v-v_{0}\right|<\eta'$ and $\left|p-p_{0}\right|<\eta''$ give
a neighborhood $U_{\sigma}$ of $\sigma$ such as $\varphi(U_{\sigma})\cap \mathbf{M}_{sing}$ is
reduced to a point, namely the point $(0,v_{0})=\varphi(\sigma)$.
\end{proof}
We now show Theorem \ref{lepremierthéorème} in the real analytic or holomorphic cases.

\begin{proof}[Proofs]

Let $(\Sigma,\Pi_\Sigma,\varphi)$ be a connected real analytic or holomorphic symplectic resolution. By construction $\varphi = (u,v)$ with $u,v$ are two 
real analytic or holomorphic functions on $\Sigma$ with values in $\Rr$ or ${\mathbb C}$. We recall that $\mathbf{M}_{sing}$ is the set $\left\{x=0\right\}\subset\mathbb{C}^{2}$.

Let $\Gamma$ be the set of the points $\sigma\in\varphi^{-1}(\mathbf{M}_{sing})$ where $v$ is regular 
(i.e. $d_{\sigma}v\neq 0$). As $\varphi$ is supposed to be surjective, the union of the set of critical values 
of $v$ with $v(\Gamma)$ is $\mathbf{M}_{sing}\cong \mathbb{R}$ or $\mathbb{C}$. By Sard's Theorem \cite{Sard1}, applied to the differentiable function 
$v$, which is surjective because $\varphi$ is surjective, the critical values of $v$ form a set of measure zero in $\mathbf{M}_{sing}\cong\Rr$ or $\mathbb{C}$. 

Moreover, for any $\sigma \in \Gamma$, there exists by Lemma \ref{hch2} a neighborhood $V_{\sigma}$ of $\sigma\in\Sigma$ such that $\varphi(V_{\sigma})\cap \mathbf{M}_{sing}$ 
is reduced to a point $(0,v_{0})=\varphi(\sigma)$. The set $\Gamma$ being a closed subset of the open subset of all regular points of $v$, it is a locally compact set and
we can extract out of any open cover of $\Gamma$ a finite or countable open cover.
In particular, one can extract from the open sets cover $(V_{\sigma})_{{\sigma}\in \Gamma }$ a finite or countable family $\left\{V_{i}\right\}_{i\in\mathbb{N}}$.
This implies that $\varphi(\Gamma)$, just as $v(\Gamma)$, is a finite or countable set.

Now, $\mathbf{M}_{sing}$ is the union of the set of critical values of $v$ with $v(\Gamma)$. But the first set is a set of measure zero (by Sard's Theorem \cite{Sard1}) 
while the second is finite or a countable set. This is impossible. This completes the proof.
\end{proof}

We now enlarge the class of Poisson manifolds  that do not admit symplectic resolutions, although they are symplectic on a dense open subset. 
We refer to \cite{c3} Chapter 5 for the definition of Poisson-Dirac submanifolds of Poisson manifolds and recall that any submanifold $N$ of a Poisson manifold $(M,\pi)$
whose tangent space $T_n N$ is in direct sum with the tangent space of the symplectic leaf ${\mathcal S}_n$ through $n$, namely such that $ T_n N \oplus T_n {\mathcal S}_n = T_n M $,
is Poisson-Dirac in a neighborhood of $n$. we also recall that any submanifold $N$ of a Poisson manifold $M$ such that $T_{n}N\oplus \pi^{\#}(T_{n}N^{\perp})=T_{n}M$ is Poisson-Dirac. Recall also that any Poisson-Dirac submanifold admits an unique induce Poisson structure called \emph{reduced Poisson structure}.

\begin{lem}\label{lem*}
Let $(M,\pi)$ be a Poisson manifold, $S$ be a symplectic leaf and $N\subset M$ a submanifold transversal at a point $n\in S\cap N$.
For every symplectic resolution $(\Sigma,\Pi_\Sigma,\varphi)$ of $(M,\pi)$, there exists a neighborhood $N'$ of $n$ in $N$ such that:
\begin{enumerate}
 \item  $\varphi^{-1}(N')$ is a submanifold of $\Sigma$,
 \item this submanifold is Poisson–Dirac,
 \item  the restriction $\varphi_{N'} : \varphi^{-1}(N')\rightarrow N'$ is a symplectic resolution
for the reduced Poisson structures of $\varphi^{-1}(N')$ and $N'$ respectively.
\end{enumerate}
\end{lem}

\begin{proof}[Proofs]
There exists, in a neighborhood $U$ of $n$,  Weinstein coordinates $(p,q,z)$ such that the submanifold $N':=N \cap U$ is given by the equations: 
$ p_1= \dots= p_{r}=q_{1}=  \dots = q_r=0$ (see\cite{c3} Chapter 1 for the definition of Weinstein's coordinate). 
The functions $p_{1},...,p_{r},q_{1},...,q_{r}$ having independent Hamiltonian vector fields at all points of $U$, their pull-back $\varphi^{*}p_{1},...,\varphi^{*}p_{r},\varphi^{*}q_{1},...,\varphi^{*}q_{r}$ also have independent Hamiltonians vector fields
at all point in $\varphi^{-1}(U)$.
Therefore, their pull-backs through $\varphi$ are independent functions. The zero locus that they define is  $\varphi^{-1}(N')$ by construction,
which is therefore a submanifold that we call $\Sigma_{N}$. This proves the first item.

A function $f$ on $M$ such that the hamiltonian vector field $X_f$ tangent to $N$ satisfies that the hamiltonian vector field $X_{\varphi^*f}$ is tangent to $\Sigma_N$. This is because if $\left\{F,q_i\right\}_{\mid_{N}}=0=\left\{F,p_i\right\}_{\mid_{N}}$ then $\left\{\varphi^*F,\varphi^*q_i\right\}_{\mid_{\Sigma_N}}=0=\left\{\varphi^*F,\varphi^*p_i\right\}_{\mid_{\Sigma_N}}$, since $\varphi$ is a Poisson morphism. Yet, $N$ and $\Sigma_N$ are sub-manifolds of Poisson-Dirac, because they are defined by the functions $(p_{1},...,p_{r},q_{1},...,q_{r})$ and $(\varphi^{*}p_{1},...,\varphi^{*}p_{r},\varphi^{*}q_{1},...,\varphi^{*}q_{r})$ respectively, whose Poisson matrix is invertible. This proves the second item.

Now, for all $F,G\in C^{\infty}(N)$, we have $\left\{\varphi^*F,\varphi^*G\right\}_{\Sigma_N}=\left\{\varphi^*\widetilde{F},\varphi^*\widetilde{G}\right\}_{\mid_{\Sigma_N}}=\varphi^*(\left\{\widetilde{F},\widetilde{G}\right\}_{\mid_{N}})$ where $\widetilde{F}$ and $\widetilde{G}$ are local extensions of $F$ and $G$. This proves that $\varphi$ is a Poisson morphism from $\Sigma_N$ to $N$. A Poisson-Direc sub-manifold of a symplectic manifold being itself symplectic, this proves third item.

\end{proof}

\begin{coro} \label{coro2} A Poisson manifold $(M,\pi)$ of dimension $2r+2$ symplectic on a dense open subset 
\begin{enumerate}
\item admits a submanifold $P$ of codimension 1 which included in to the singular locus of $\pi_{M}$,
\item has at least one point $n\in P$ where the rank of $\pi$ is $2r$,
\end{enumerate} 
does not admit a proper symplectic resolution. In the real analytic or holomorphic cases, under the same assumptions it does not admit a symplectic connected resolution.
\end{coro}

\begin{proof}[Proofs]
Let $F$ be a local function in a neighborhood of $n$. Let us prove that the Hamiltonian vector field 
$X_F$ of $F$ is tangent to $P$. Since $P$ has codimension $1$, if there exists a point $n' \in P$ where $X_F$ is not tangent to $P$,
every point in a neighborhood of $n'$ in $M$ is obtained from a point of $P$  by following the flow of $X_F$. But since every point in $P$ is a singular point for $\pi_M$,
and since the flow of $X_F$ is made of Poisson diffeomorphisms, this would imply that every point in a neighborhood of $n'$ in $M$ is a singular point, which contradicts the assumption on $\pi_M$.

Let $N$ be any submanifold of dimension $2$ transversal to the symplectic leaf ${\mathcal S}_n$ through $n$ where $n\in P$ is a point where the rank of $\pi_{M}$ is $2r$. By the previous point $S$ is a submanifold of $P$. 
Since $N$ is transversal to $S_{n}$, $P \cap N$
is a curve in $N$, at least in a neighborhood of $n\in N$. We denote by $P_{N}$ this curve.

Suppose that $(M,\pi_M)$ admits a proper symplectic resolution. Then, 
by Lemma \ref{lem*}, there is a neighborhood $N'$ of $n$ in $N$
that does admit a proper symplectic resolution.

But upon shrinking $N'$ if necessary, we can assume that $N'$
is diffeomorphic to an open subset of ${\mathbb R}^2$, isomorphism under which the curve $P_{N}$ is given by the equation $x=0$,
with $x,y$ being the canonical coordinates on ${\mathbb R}^2$.
The Poisson structure on $N'$ being now of the form (\ref{laplussimple}), it does
not admit  proper symplectic resolutions by Theorem \ref{lepremierthéorème}, which contradicts the previous statement. Hence $(M,\pi_M)$
does not admit proper symplectic resolutions. The real analytic or holomorphic cases are similar.
\end{proof}

\begin{lem} \label{lem} Let $(M,\pi)$ be a Poisson manifold of dimension $n\geq 4$, such that the singular locus of $\pi$ contains 
a submanifold $P\subset M$ of codimension $1$. If a symplectic resolution exists, then the bivector $\pi$ can not be zero at all points in $P$.
\end{lem}     
\begin{proof}[Proofs] Consider a projection $\psi: U\rightarrow P$ where $U$ is a neighborhood of $P$ in $M$.
Let $(\Sigma,\Pi,\varphi)$ be a symplectic resolution of $(M,\pi)$. The map $\psi \circ \varphi : \varphi^{-1}(U)\rightarrow P$ is surjective because both 
$\varphi$ and $\psi$ are surjective. 
By Sard's Theorem, there exists regular value for $ \psi \circ \varphi$, i.e. there exists point $ p \in P$
such that for every $\sigma \in \varphi^{-1}(U) \subset \Sigma$, the composition
 $T_{\varphi (\sigma)} \psi\circ T_\sigma \varphi$ is a surjective linear map.
Since $\varphi$ is surjective, there exists $\sigma \in \Sigma$
such that  $\varphi ( \sigma) =p $. Such a point belongs to $\varphi^{-1}(U)$ by construction, and 
satisfies that $T_{\varphi (\sigma)} \psi\circ T_\sigma \varphi$ is a surjective linear map.
In turn, this implies that
for any choice $x_1, \dots,x_{n-1}$ of local coordinates on $P$ around $M$, 
the functions  $(\psi \circ \varphi)^* x_1, \dots,(\psi \circ \varphi)^* x_{n-1}$   are linearly independent at the point $\sigma \in \varphi^{-1}(U)  $.

Let us compute the Poisson brackets of these functions at the point $\sigma$. Since the map $\varphi$ is a Poisson morphism, for all $i,j =1, \dots,n-1$:
\begin{eqnarray*}
\left\{(\psi \circ \varphi)^* x_i,(\psi \circ \varphi)^* x_j \right\}_{\Sigma}(\sigma) &=&
\left\{\varphi^{*} \psi^* x_{i},\varphi^{*} \psi^* x_{j}\right\}_{\Sigma}(\sigma) \\
 &=&
\left\{\psi^* x_{i},\psi^* x_j \right\}_{P}(\varphi(\sigma))=  0  
\end{eqnarray*}
But it is impossible to have $n-1$ independent functions on a symplectic manifold whose brackets are equal to zero at a given point. This completes the proof.
\end{proof}

\begin{thm}
\label{theo:final}
A Poisson manifold $(M,\pi)$ symplectic on an open dense subset which contains a submanifold of codimension $1$ of singular points for $\pi_{M}$, does not admit a proper symplectic resolution. When $(M,\pi_{M})$ is a holomorphic Poisson manifold no connected symplectic resolution exists.
\end{thm}

\begin{proof}[Proofs]
We prove the theorem by induction on $\frac{1}{2}\dim(M)$. If $\frac{1}{2}\dim(M)=1$, the theorem reduces to the statement of Theorem \ref{lepremierthéorème} in local coordinates. We assume that the theorem is true until  $n=\frac{1}{2}\dim(M)$
and we show it for $n+1=\frac{1}{2}\dim(M)$. We assume that the singular locus of $\pi$ contains a submanifold $P$ of codimension $1$. By Lemma \ref{lem}, the rank of $\pi_{M}$ on $P$ can not be zero. Let $m$ be a point where in $P$ where $\pi_{m}\neq 0$. 
\noindent If a symplectic resolution $(\Sigma,\pi_{\Sigma},\varphi)$ of $(M,\pi_{M})$ exists, by Lemma \ref{lem*}, the inverse image of 
any submanifold $N$ transverse to the symplectic leaf through $m$ by $\varphi$ is a symplectic submanifold of $\Sigma$ and the restriction $\varphi_{N} : \varphi^{-1}(N)\rightarrow N$ is a 
symplectic resolution of dimension $2(n+1-r)$ with $2r$ the rank of $\pi_{M}$ at $m$. Since the singular locus of $\pi$ contains a submanifold of codimension $1$, we get a contradiction with the induction hypothesis. This shows the result.
\end{proof}

\begin{coro} A holomorphic Poisson manifold does not admit connected symplectic resolutions unless it is symplectic.
\end{coro}

\begin{proof}[Proofs]  A Poisson manifold $(M,\pi_{M})$ that admits a symplectic realization must be symplectic on an open dense subset by Proposition \ref{nari}. 
For such a manifold, the multivector field $\pi^n=\pi\wedge\cdots\wedge\pi$ ($n$ times, where $2n$ is the dimension of $M$) is a section of a vector bundle of rank $1$,
namely $\wedge^{2n}TM$ which vanishes precisely at critical points. If it vanishes in at least one point, i.e, if $(M,\pi_M)$ is not symplectic, 
there is also a submanifold of codimension $1$ where it vanishes (by Weierstrass preparation theorem - which implies that the zero locus of any holomorphic function
contains regular points, around which it is simply a submanifold of codimenion $1$).
Theorem \ref{theo:final} allows to conclude that no symplectic realization exist.

\end{proof}


\begin{thebibliography}{}

\end{thebibliography}


\begin{thebibliography}{777}
 \bibliographystyle{alpha}

 
\bibitem{sym} Arnaud Beauville, Symplectic singularities, \emph{Invent. Math.}, {\bf 139}, (2000), 541--549.
 
\bibitem{BellamySchedler} Travis Schedler, A new linear quotient of ${\bf C}^4$ admitting a symplectic
   resolution,
   \emph{Math. Z.},
   {\bf 273},
  (2013), {753--769}.
 
\bibitem{BX16} Damien Broka, Ping Xu,
\emph{Symplectic Realizations of Holomorphic Poisson Manifolds}, arXiv. (2015).

\bibitem{DCW} Alain Coste, Pierre Dazord, and Alan Weinstein, Groupo\"ides symplectiques, \emph{Publications du Département de mathématiques de Lyon}, {\bf 2A},(1987): 1--62. 

\bibitem{DZ} Jean-Paul Dufour, Nguyen Tien Zung, \emph{ Poisson Structures and Their Normal Forms},
Springer Birkhäuser, (2005) 

\bibitem{res} BaoHua Fu, Symplectic Resolutions for Nilpotent Orbits, 
\emph{ Invent. math.}, {\bf 151}, (2003), 157--186. 
 
 
\bibitem{c1234} Camille Laurent-Gengoux,
\emph{From Lie groupoids to resolutions of singularities. Applications to sympletic and Poisson resolution },
arXiv. (2007). 
 
\bibitem{c3} Camille Laurent-Gengoux, Anne Pichereau, Pol Vanhaecke,
\emph{Poisson Structures}, Spinger, (2013).


\bibitem{c2} Andr\'e Lichnerowicz, Vari\'et\'e de Poisson et leurs alg\`ebres, \emph{J. Diff Geom.}, {\bf 12}, (1977), 253--300. 


\bibitem{c5} John W. Milnor, \emph{Topology from the Differentiable Viewpoint}, Princeton University Press, (1997). 

\bibitem{c7312} Alan Weinstein, The local structure of Poisson manifolds, \emph{J. Diff Geom.}, {\bf 18}, (1983), 523--557. 

\bibitem{Sard1} John W. Milnor. \emph{Topology from the differential viewpoint}, The University Press of Virginia, (1965).

\bibitem{lemarie} Caroline Lemari\'e. \emph{Quelques structures de Poisson et \'equations de Lax associ\'ees au r\'eseau de Toeplitz et au r\'eseau de Schur}, PhD thesis in mathematics under the direction of Pol Vanhaecke, Universit\'e de Poitiers, (2012).



\end{thebibliography}
\end{document}